\documentclass[a4paper]{amsart}
\usepackage[utf8]{inputenc}
\usepackage[english]{babel}
\usepackage{fancyhdr}
\usepackage{amsmath}
\usepackage{amsfonts,amstext}
\usepackage{amssymb,amsthm,amscd,amsxtra,wasysym,graphicx}
\usepackage{grffile}
\usepackage{mathrsfs,esint,comment}
\usepackage{tikz-cd}
\usepackage{enumitem,mathtools,dsfont}
\usepackage{palatino,mathpazo}
\usepackage{xcolor}

\def\Bibtex{{\rm B\kern-.05em{\sc i\kern-.025em b}\kern-0.08em T\kern-.1667em\lower.7ex\hbox{E}\kern-.125emX}}
 \usepackage{hyperref}
\hypersetup{linktocpage=true,
	unicode=false,          
	pdftoolbar=true,       
	pdfmenubar=true,       
	pdffitwindow=false,     
	pdfstartview={FitH}, 
	pdftitle={Monge-Amp\`ere equations},   
	pdfauthor={Quang-Tuan Dang},   
	colorlinks=true,  
	linkcolor=purple,         
	citecolor=blue,        
	filecolor=green,
	urlcolor=blue}          

\usepackage{cleveref}
\usepackage{indentfirst}
\theoremstyle{plain}
\newtheorem{theorem}{Theorem}
\newtheorem{lemma}[theorem]{Lemma}

\newtheorem{proposition}[theorem]{Proposition}
\newtheorem{example}[theorem]{Example}

\theoremstyle{definition}
\newtheorem{definition}[theorem]{Definition}
\newtheorem{remark}[theorem]{Remark}

\theoremstyle{plain}
\newtheorem{bigthmm}{Theorem}

\numberwithin{theorem}{section}
\numberwithin{equation}{section}
 
\DeclareMathOperator \PSH {{\rm PSH}}

\DeclareMathOperator \MA {{\rm MA}}

\DeclareMathOperator \Vol {{\rm Vol}}

\DeclareMathOperator \Capa {{\rm Cap}}

\def\f{\varphi}
\def\dc{dd^c}

\begin{document}
	
	\title[Uniform estimates for complex Monge--Amp\`ere equations]{Uniform estimates for complex Monge-Amp\`ere equations: big cohomology classes}
	\author{Quang-Tuan Dang, Lei Zhang and Bin Zhou}
		\address{Yau Mathmatical Sciences Center, Tsinghua University, Beiing 100084}
		\email{dangquangtuan10@gmail.com $\&$ dangqt@mail.tsinghua.edu.cn}
        \address{Yau Mathmatical Sciences Center, Tsinghua University, Beiing 100084}
        \email{leizhang92@mail.tsinghua.edu.cn}
        \address{School of Mathematical Sciences, Peking University, Beijing 100871}
        \email{bzhou@pku.edu.cn}
	\date{\today}
	\keywords{Complex Monge--Amp\`ere equations, uniform a priori estimates, big cohomology classes, Moser-Trudinger inequalities}
	\subjclass[2020]{32U20, 32W20, 32U05}

	\begin{abstract} We prove uniform a priori estimates for solutions to degenerate complex Monge--Amp\`ere equations in big cohomology classes, using both auxiliary-function technique developed by Guo, Phong and Tong [On $L^\infty$-estimates for complex Monge-Amp\`ere equations, Ann. of Math. (2) 198 (2023), no.1, 393-418], and quasi-psh envelope approach developed by Guedj and Lu [Quasi-plurisubharmonic envelopes 1: Uniform estimates on K\"ahler manifolds, J. Eur. Math. Soc. (JEMS) 27 (2025), no. 3, 1185-1208.]. As an application, we apply our method to prove the Moser-Trudinger and Brezis-Merle-type inequalities for complex Monge-Amp\`ere equations.
	\end{abstract}
	\maketitle
	\tableofcontents

	\section{Introduction}\label{sect: intro}


Over the last few decades, finding canonical metrics on complex varieties has been a fundamental problem in complex geometry.
Since Yau's solution to the Calabi conjecture~\cite{yau1978ricci}, complex Monge--Ampère equations have become one of the most  powerful tools in K\"ahler geometry. A crucial step in order to prove the existence of a solution to
such equations is to establish a uniform a priori estimate.

 As evidenced by recent developments in K\"ahler geometry in connection with the Minimal Model Program, it is natural and necessary to allow the underlying varieties in question to be singular. They led to the study of degenerate complex Monge--Amp\`ere equations and the construction of singular K\"ahler--Einstein metrics; see, e.g.,~\cite{eyssidieux2009singular,demailly2010degenerate, song2012canonical,tosatti2010adiabatic,berman2013variational,berman2014kahler,berman2019kahler} and references therein.
A major breakthrough in the field was achieved by Ko\l odziej \cite{kolodziej1998complex}, who proved a uniform a priori estimate for solutions of degenerate complex Monge--Amp\`ere equations on compact K\"ahler manifolds assuming only weak integrability. This result laid the foundation for the modern pluripotential-theoretic approach to complex Monge--Amp\`ere equations. Since then, the pursuit of optimal and versatile approaches to such uniform estimates has remained an active and fertile area of research.

During the last few years, there has been renewed interest in understanding the mechanism behind Ko\l odziej-type estimates and in obtaining sharper quantitative versions. A remarkable development in this direction is the introduction by Guo, Phong, and Tong ~\cite{GuoPhongTong23-estimates} of auxiliary-function techniques that provide a new route to uniform estimates and Moser--Trudinger-type inequalities for complex Monge--Ampère equations, which is inspired by the work of Chen and Cheng~\cite{ChenCheng21-csck1,ChenCheng21-csck2}, and  Wang, Wang, and Zhou~\cite{WangWangZhou21-estimate}. We also refer the interested readers to \cite{GuoPhongTong23-estimates,Guo-Phong2024-estimates,GuoPhong23-auxiliary,GuoPhongTong25-auxiliary,Qiao2025-sharp} for further directions. {However, it is not easy to directly apply their approach to the case of big cohomology classes.}

At the same time, Guedj and Lu~\cite{guedj2021quasi} developed a novel envelope-theoretic approach to degenerate complex Monge--Ampère equations. Their theory of quasi-plurisubharmonic envelopes~\cite{guedj2021quasi1,guedj2022quasi,guedj2021quasi} provides a flexible framework for studying Monge--Amp\`ere equations in both K\"ahler and Hermitian settings and has led to several important advances in pluripotential theory. While this approach is particularly well suited to the degenerate setting and yields robust qualitative results, it generally provides less explicit quantitative estimates than the auxiliary-function method.

The main goal of this paper is to bridge the auxiliary-function method of Guo, Phong, and Tong with the envelope techniques of Guedj and Lu in the context of big cohomology classes. We show that these two methods can be effectively combined in the setting of big cohomology classes. By exploiting the strengths of both methods, we obtain a new proof of uniform estimates for solutions of degenerate complex Monge--Ampère equations in big classes. Furthermore, our approach yields refined quantitative integrability and energy estimates for the corresponding solutions, including Brezis--Merle and Moser--Trudinger type inequalities.

\medskip
To state our main results, we fix some notation and terminology. Let $X$ be a $n$-dimensional compact K\"ahler manifold equipped with a K\"ahler metric $\omega_X$, normalized by $\int_X\omega_X^n=1$. Let $dV$ denote the Lebesgue measure on $X$. 
Let $\theta$ be a closed smooth real $(1,1)$-form.
	  We say that $u$ is $\theta$-plurisubharmonic ($\theta$-psh)  if it is locally the sum of a psh and a smooth function and satisfies $\theta_u:=\theta+\dc u\geq 0$ in the weak sense of currents.  
    We let $\PSH(X, \theta)$ denote the set of $\theta$-psh functions on $X$. Recall that
	the cohomology class $\{\theta\}$ is {\em big} if there exists $\rho\in\PSH(X,\theta)$ such that $\theta+\dc\rho\geq \delta{\omega_X}$ for some small constant $\delta>0$. Let $\Vol(\theta)$ denote the volume of a big cohomology class $\{\theta\}$. 

\medskip
We are interested in studying  the complex Monge--Amp\`ere equation type \begin{equation}\label{eq: dcmae}
		\MA_\theta(\varphi):=\frac{1}{\Vol(\theta)}(\theta+\dc u)^n=\mu,\quad\sup_X \varphi=0, 
	\end{equation}
	where
	$\mu$ is a probability measure on $X$ that puts no mass on pluripolar subsets, $\varphi$ is the unknown $\theta$-{psh} function, and the left-hand side of~\eqref{eq: dcmae} denotes the non-pluripolar Monge--Amp\`ere measure, constructed in ~\cite{bedford1987fine,boucksom2010monge}; cf. Section~\ref{sect: nonpluripolar}. We let $\mathcal{E}(X,\theta)$ denote the set of $\theta$-psh functions with {\em full Monge-Amp\`ere mass}, i.e., $\int_X\MA_\theta(\varphi)=1$.

\medskip 

Our first main result establishes a uniform estimate for solutions of Monge-Ampère equations with densities belonging to a large class of Orlicz spaces defined by a weight function satisfying \hyperlink{conditionK}{condition (K)}. 
This generalizes the classical assumptions of $L^{1+\varepsilon}$-integrability and aligns with the most recent developments in the field.
\begin{bigthmm}\label{main} Let $(X,\omega_X)$ be a compact K\"ahler manifold. Let $\theta$ be a real smooth closed (1,1) form representing a big cohomology class. Assume that $\theta\leq A\omega_X$ for some $A>0$.
    Let $\mu= fdV$ be a probability measure on $X$, where $dV$ is the Lebesgue measure and $f\in L^w(dV)$ with density $w$ satisfying \hyperlink{conditionK}{condition~(K)}. 
    Then there exists a unique solution $\varphi\in\mathcal{E}(X,\theta)$ with minimal singularities to the following equation \begin{equation}\label{cmae-big}
    \MA_\theta(\varphi)=\mu,\quad \sup_X\varphi=-1.
\end{equation}
    More precisely, there is a constant $C$ depending on $\omega_X$, $A$, $dV$, $\theta$, $n$, and $\|f\|_w$ such that 
    \[ \varphi\geq V_\theta-C.\]
\end{bigthmm}
Here, $V_\theta=\sup\{u\in\PSH(X,\theta):u\leq 0\}$ denotes the $\theta$-psh function with minimal singularities and $\|f\|_w$ denotes the Luxembourg norm of $f$ in the Orlicz sapce $L^w(dV)$, cf. Definition~\ref{def: Orlicz-norm}. The existence of solutions in the finite energy class $\mathcal{E}(X,\theta)$ was first established in the pioneering works of Guedj and Zeriahi~\cite{guedj2007weighted} and Boucksom, Eyssidieux, Guedj and Zeriahi~\cite{boucksom2010monge}. An alternative proof based on variational methods was later developed by Berman, Boucksom, Guedj, and Zeriahi~\cite{berman2013variational}. The uniqueness of solutions in $\mathcal{E}(X,\theta)$ follows from the work of Dinew~\cite{dinew2009uniqueness}. Building on Ko\l odziej's uniform estimate~\cite{kolodziej1998complex}, Boucksom, Eyssidieux, Guedj, and Zeriahi proved that the solution has minimal singularities whenever the density of the right-hand side belongs to a suitable Orlicz space; see \cite[Theorem 4.1 and Remark 4.6]{boucksom2010monge}. {We refer the reader to~\cite{demailly2010degenerate,darvas2021log,GGZ23-families} for further developments and related results. Very recently, Zhang and Zhang~\cite{ZhangZhang26} established a uniform estimate under the assumption of density in Orlicz spaces following the Ko\l odziej approach.} We emphasize that B\l ocki has provided a different approach in \cite{blocki2011uniform} based on the Alexandroff-Bakelman-Pucci maximum principle, requiring the reference form $\theta$ be Hermitian. His method has been applied to many contexts, including subsolutions~\cite{Szekelyhidi18-fully,Phong-To21-fully}, equations with gradient terms~\cite{Tosatti-Weinkove21-gradient}.

The main novelty of this work is the combination of two recently developed techniques. 
Rather than estimating the Monge--Ampère capacities of sublevel sets as in the classical Ko\l odziej theory, we construct suitable auxiliary functions following Guo, Phong and Tong~\cite{GuoPhongTong23-estimates} and exploit the quasi-psh envelope method of Guedj and Lu \cite{guedj2021quasi}. This allows us to obtain uniform and quantitative estimates directly at the level of the complex Monge--Ampère equation. 
Our approach thus extends with minor modifications to the Hermitian in the companion paper~\cite{DangZhangZhou26-estimate-hermit}.

    \medskip 
Our second result establishes Brezis--Merle type integrability estimates for solutions. Such estimates can be viewed as nonlinear analogs of the classical Brezis--Merle theorem and provide refined information on the singularity behavior of solutions.

     \begin{bigthmm}\label{main-Brezis-Merle}  Assume that $\mu= f\omega_X^n$ is a probability measure on $X$ and $f\in L^w(dV)$, where $w(t)=t(\log(1+t))^p$ for $p\geq 0$.
       Let $\varphi\in\mathcal{E}(X,\theta)$ be a solution to \eqref{cmae-big}. Then there exist constants $c>0$ and $C>0$ depending on  $\omega_X$, $n$, $p$, $A$, $c$, and $\|f\|_w$ such that
       \begin{itemize}
           \item For $p\in [0,n)$, we have
       \begin{equation}\label{eq0: Moser-Trudinger}
           \int_X e^{c(V_\theta-\varphi)^{\frac{n}{n-p}}}\omega_X^n\leq C.
       \end{equation} 
       \item for $p\geq  n$, we have for any $N>0 $ 
       \[ \int_Xe^{c(V_\theta-\varphi)^N}\omega_X^n\leq C.\]    
       \end{itemize}    
       Moreover, we also have the energy-like inequalities 
       \[\int(V_\theta-\varphi)^{r} f\omega_X^n\leq C, \]
       where $r=\frac{np}{n-p}$ if $p\in(0,n)$, and $r=N$ any positive constant if $p\geq n$.
   \end{bigthmm}

The estimates obtained here extend several recent results of Guo and Phong \cite{GuoPhong24-entropy} from the K\"ahler setting to the much more general framework of big cohomology classes. They also complement the envelope approach of Di Nezza, Guedj, and Lu~\cite{di2021finite} (see also~\cite{DiNezza24-entropy}) by showing that the latter naturally accommodates the auxiliary-function method. In particular, inspired by J. Liu~\cite{liu2024relative}, our work provides a unified framework in which uniform estimates, exponential integrability, and energy inequalities can be derived simultaneously. We refer interested readers to~\cite{cegrell19-measure,BermanBerndtsson14-symm,BermanBerndtsson22-Moser,
Ahag-Czyz19-MT,WangWangZhou20-Trudinger,di2021finite,DinhMarinescuVu23-Trudinger} for the local version.

\medskip

\noindent\textbf{Main steps of the proof.} We will describe the outline of the proof of Theorem~\ref{main} concerning uniform a priori estimates.

 \begin{itemize}
     \item {\it Step 1. Comparison with auxiliary {envelopes}.} We define the sub-level set $\Omega_s:=\{\varphi<V_\theta-s\}$ for $s>0$. 
Fix $r>0$. We solve an auxiliary complex Monge--Amp\`ere equation 
\begin{equation*}
    \MA_\theta(v_s)=\frac{\mathbf{1}_{\Omega_s}(V_\theta-\varphi-s)^r }{A_s}fdV,\; v_s\in\mathcal{E}(X,\theta),\quad\sup_X v_s=-1,
\end{equation*} where $A_s=\int_{\Omega_s}(V_\theta-\varphi-s)^rfdV$. The existence of solution $v_s$ is shown in~\cite{boucksom2010monge,berman2013variational}.
The key point is to compare $v_s$ with the original solution $\varphi$. Using Lemma~\ref{lem: key}, we establish the following
    \begin{equation}\label{eq: key-equation0}
       \frac{n}{n+r} (2A_s)^{-\frac{1}{n}}(V_\theta-\varphi-s)^{\frac{n+r}{n}}\leq V_\theta-v_s+\frac{n}{n+r}(2A_s)^{\frac{1}{r}}.
    \end{equation}
When $\theta$ is semi-positive, estimate \eqref{eq: key-equation0} is precisely the key inequality established in \cite[Lemma 1]{GuoPhongTong23-estimates}. A major difficulty in extending this approach to big cohomology classes stems from the singular nature of $\theta$-psh functions with minimal singularities. Consequently, the techniques of \cite{GuoPhongTong23-estimates} are no longer directly applicable. Our main innovation is to combine the auxiliary-function method with the quasi-psh envelope techniques, thus overcoming the difficulties caused by these singularities.
    

     \item {\it Step 2. Treating the case of densities in $L^{1+\varepsilon}$.} 
 From the estimate~\eqref{eq: key-equation0}, we can follow the argument of~\cite{GuoPhongTong23-estimates} to obtain
     \[t\phi(s+t)\leq C_0\phi(s)^{1+\delta_0}, \;\forall\, t\in [0,1],\; s>0, \]
where $\phi(s):=\mu(\Omega_s)$, $\delta_0>0$ depends on $n$ and $\varepsilon$, while $C_0$ depends only on $n$, $\omega_X$, $dV$, $A$, $\varepsilon$ and $\|f\|_{1+\varepsilon}$. By a standard lemma of De Giorgi, this inequality yields the desired uniform estimate.

     \item {\it Step 3. Reduction to the $L^{1+\varepsilon}$ case and conclusion.}  We follow the trick in~\cite[Theorem 2.1]{GuedjLu25-estimate} and in~\cite[Theorem 3.1]{liu2024relative} to reduce to the case of densities in $L^{1+\varepsilon}$.
 \end{itemize}

\medskip

\noindent\textbf{Organization of the paper.}
In Section~\ref{sect: recap}, we recall the necessary preliminaries on pluripotential theory, quasi-psh envelopes, and the relevant function spaces. Section~~\ref{sect: estimate} contains the proof of the uniform estimate, while Section~\ref{sect: inequality} is devoted to the Brezis--Merle-type estimates and the corresponding energy inequalities. 

\medskip 
\noindent\textbf{Acknowledgments.} Q.-T. Dang is supported by  the Shuimu Scholar program of Tsinghua University. L. Zhang is supported by Postdoctoral Fellowship Program of China  GZC20240867.
B. Zhou is partially supported by  National Key R$\&$D Program of China 2023YFA009900 and NSFC  Grant 12271008.

   \subsection*{Ethics declarations} The authors declare no conflict of interest.
	\section{Recap on pluripotential theory}\label{sect: recap}
Let $X$ be a $n$-dimensional compact K\"ahler manifold equipped with a K\"ahler metric $\omega_X$, normalized by $\int_X\omega_X^n=1$. Let $\theta$ be a closed smooth real $(1,1)$-form. In the whole article, we assume that $\theta\leq A\omega_X$ for $A>0$.
    
\subsection{Quasi-plurisubharmonic functions}
  A function $u:X\to\mathbb{R}\cup\{-\infty\}$ is called \emph{quasi-plurisubharmonic
	(qpsh)} if it can be locally written as the sum of a plurisubharmonic function and a smooth function. 
	We say that $u$ is \emph{$\theta$-plurisubharmonic ($\theta$-psh)}  if it is qpsh and $\theta_u:=\theta+\dc u\geq 0$ in the sense of currents.  We let $\PSH(X, \theta)$ denote the set of $\theta$-psh functions on $X$.

The set $\PSH(X,\theta)$ is endowed with 
	the $L^r(X,\omega_X^n)$-topology for all $r\geq 1$. By Hartogs' lemma, $\f\mapsto\sup_X \f$ is continuous in this weak topology. In particular, the set of $\f\in\PSH(X,\theta)$, with $\sup_X \f=0$ is compact in $L^r(X,\omega_X^n)$. 
	We refer the reader to~\cite{demaillycomplex,guedj2017degenerate} for basic properties of $\theta$-psh functions. It also enjoys uniform integrability properties.
    {We observe that if a function $u\in\PSH(X,\theta)$ then $u\in\PSH(X,A\omega_X)$. 

\begin{theorem}[{\cite{skoda1972sous,tian1987kahler,guedj2017degenerate}}]\label{thm: Skoda/Tian}
    There exists $\alpha=\alpha(n,A\omega_X)>0$ such that for all $\varphi\in\PSH(X,\theta)$,
    \begin{equation}\label{eq: alpha}
        \int_Xe^{-\alpha(\varphi-\sup_X\varphi)}\omega_X^n\leq C,
    \end{equation}
    where $C=C(n,A,\alpha)$ is independent of $\varphi$ and $\theta$. Moreover, for all $\varphi\in\PSH(X,\theta)$, 
    \[\int_X(\varphi-\sup_X\varphi)\omega_X^n\geq -C. \]
\end{theorem}}

    \begin{definition}
        We say that
	the cohomology class $\{\theta\}$ is {\em big} if there exists $\rho\in\PSH(X,\theta)$ such that $\theta+\dc\rho\geq \delta{\omega_X}$ for some small constant $\delta>0$.

    The {\em ample locus} $\textrm{Amp}(\theta)$ is the Zariski open subset of points $x\in X$ such that there exists
a K\"ahler current $\theta+\dc\rho $, which is smooth in a neighborhood of $x$. 
    \end{definition}

     Assume $u,v\in \PSH(X,\theta)$. We say that $u$ is {\it less (resp. more) singular} than $v$, and denote by $u\succeq v$ (resp. $u \preceq v$), if there exists a constant $C$ such that $u+C\geq v$ (resp. $u\leq v+C$) on $X$. We say that $u$, $v$ have the {\em same singularity type}, and denote by $u\simeq v$ if $u \preceq v$ and $u\succeq v$.
A $\theta$-psh function $u$ is said to have {\em minimal singularities} if it is less singular than any other $\theta$-psh function. Following Demailly,
we introduce the extremal function $V_\theta$ defined by
\[ V_\theta(x):=\sup\{\varphi(x)\in\PSH(X,\theta): \varphi\leq 0 \}.\]
One can see that $V_\theta$ is a $\theta$-psh function with minimal singularities. 
In particular, $\theta$ is semi-positive if and only if $V_\theta=0$.

\medskip 
Recall that a  Borel set $E\subset X$ is {\em (locally) puripolar} if for each $x\in X$, there exists an open neighborhood $U$ of $x$ and a psh function $u$ on $U$ such that $E\cap U\subset \{u=-\infty\}$. 
	As follows from~\cite[Theorem 12.5]{guedj2017degenerate}, the set $E$ is globally pluripolar, i.e., there exists  $u\in \PSH(X,\omega_X)$ such that $E\subset\{u=-\infty \}$. If $\theta$ is big, then there exists a function $\rho\in \PSH(X,\theta)$ such that $\theta+\dc\rho\geq \varepsilon_0\omega_X$.
    The function $u':=\varepsilon_0 u+\rho$ is $\theta$-psh and its $-\infty$-locus contains $E$.
    
\subsection{Non-pluripolar measures}\label{sect: nonpluripolar}
Let $\theta_1,\ldots,\theta_p$ be smooth closed real (1,1)-forms for $1\leq p\le n$. 
Let $u_j$ be $\theta_j$-psh functions for $j\in\{1,\ldots,p\}$ and put $\theta_{j,u_j}:=\theta_j +\dc u_j$. We recall how to define the non-pluripolar product $\theta_{1,u_1} \wedge \cdots \wedge \theta_{p,u_p}$. We write locally $\theta_{j,u_j}= \dc v_j$, where $v_j$ is psh. By \cite{bedford1987fine, boucksom2010monge}, one knows that the sequence of positive currents 
$$\mathbf{1}_{\cap_{j=1}^p \{v_j>-k\}}\dc \max\{v_1, -k\} \wedge \cdots \wedge \dc \max\{v_p, -k\}$$ 
is increasing in $k \in \mathbb{N}$ and converges to a closed positive current, which is independent of the choice of local potentials $v_j$'s. Thus, we obtain a well-defined closed positive current on $X$ which is called the \emph{non-pluripolar product} $\theta_{1,u_1} \wedge \cdots \wedge \theta_{p,u_p}$ of $\theta_{1,u_1}, \ldots, \theta_{p,u_p}$.  
For any $u\in\PSH(X,\theta)$, the {\em non-pluripolar complex Monge--Amp\`ere measure} of $u$ is given by \[\theta_u^n:= (\theta+\dc u)^n.\] We observe that $\theta_u^n$ puts no mass on pluripolar sets.
	The volume of a big class $\{\theta\}$ is defined by
 \[\Vol(\{\theta\}):=\int_X\theta_{V_\theta}^n. \]
A $\theta$-psh function $u$ is said to have {\em full Monge--Amp\`ere mass} if $\int_X\theta_u^n=\Vol(\theta)$.
We note that by~\cite[Theorem 4.7]{boucksom2002volume}, the class $\{\theta\}$ is big if and only if $\Vol(\theta)>0$.
When $\{\theta\}$ is big, we define the non-pluripolar Monge--Amp\`ere measure for $\varphi\in\PSH(X,\theta)$:
\[ \MA_\theta(\varphi):=\frac{1}{\Vol(\theta)}(\theta+\dc\varphi)^n,\]
which is a probability measure on $X$.

Given a potential $\phi\in\PSH(X,\theta)$, we let $\PSH(X,\theta,\phi)$ denote the set of $\theta$-psh functions $u$ such that $u\leq\phi$. We also define by $\mathcal{E}(X,\theta,\phi)$ the set of $u\in\PSH(X,\theta,\phi)$ of full Monge--Amp\`ere mass with respect to $\phi$, i.e.,
	$\int_X\theta_u^n=\int_X\theta_\phi^n$.
	When $\phi=V_\theta$ we simply denote by $\mathcal{E}(X,\theta)$.

\medskip 

We recall here the plurifine locality of the non-pluripolar Monge--Amp\`ere product (cf. \cite[Corollary 4.3]{bedford1987fine} or \cite[Section 1.2]{boucksom2010monge}) for later use. This topology is the coarsest such that all quasi-psh functions with values in $\mathbb{R}$ are continuous.
\begin{lemma}[{\cite[Corollay 4.3]{bedford1987fine}}]\label{lem: plurifine}
	Assume that $\varphi$, $\psi$ are $\theta$-psh functions such that $\varphi=\psi$ on an open set $U$ in the plurifine topology. Then 
	\begin{equation*}
	\mathbf{1}_U\theta_\varphi^n=\mathbf{1}_U\theta_\psi^n.
	\end{equation*}
\end{lemma}
For practice, we stress that sets of the form
$\{u < v\}$, where $u$ and $v$ are quasi-psh functions, are open in the plurifine topology. Lemma~\ref{lem: plurifine} will be referred to as the {\em plurifine locality property}.

 We recall the following classical inequality.
\begin{lemma}\label{lem: maxprin}
	Let $\varphi,\psi\in\PSH(X,\theta)$. Then
	\[\theta_{\max(\varphi,\psi)}^n\geq \mathbf{1}_{\{\psi\leq\varphi \}}\theta_\varphi^n+\mathbf{1}_{\{\varphi<\psi\}}\theta_{\psi}^n. \]
    In particular, if $\varphi\leq\psi$ then $\mathbf{1}_{\{\varphi=\psi\}}\theta^n_\varphi\leq \mathbf{1}_{\{\varphi=\psi\}}\theta^n_\psi$.
\end{lemma}\begin{proof}
    See e.g.,~{\cite[Lemma 2.9]{darvas2025relative}}.
\end{proof}

    \subsection{Quasi-envelopes}
Given a measurable function $f:X\to\mathbb{R}$, we define the {\em $\theta$-psh envelope} of $f$ by
\begin{equation*}
P_\theta(f):=(\sup\{u\in\PSH(X,\theta): u\leq f\;\text{on}\, X \})^*,
\end{equation*} where the star means that we take the upper semi-continuous regularization. We use the convention that $\sup \varnothing=-\infty$.

We observe that $P_\theta(f)\in\PSH(X,\theta)$ if and only if there exists a function $u\in\PSH(X,\theta)$ lying below $f$. We also mention that $P_\theta(f+C)=P_\theta(f)+C$ for any constant $C$.

By~\cite[Chapter 9]{guedj2017degenerate}, if $f$ is finite on a non-pluripolar set, then $P_\theta(f)$ is a well-defined $\theta$-psh function on $X$. 

If $f=-\mathbf{1}_E$ is the negative characteristic function of a subset $E$, then $P_\theta(f)=h^*_{E,\theta}$ is the so-called relative extremal function of $E$, cf.~\cite[Chapter 9]{guedj2017degenerate}. 
When $f=0$ then $P_\theta(0)=V_\theta$, which is a $\theta$-psh function with minimal singularities.


Given a $\theta$-psh function $\phi$, 
Ross and Witt Nystr\"om~\cite{ross2014analytic} introduced the ``rooftop envelope" as follows
\[P_\theta[\phi](f)= \left( \lim_{C\to +\infty}P_\theta(\min(\phi+C,f))\right)^*. \]
When $f=0$ we simply write $P_\theta[\phi]$. 
\begin{definition}
A function $\phi\in\PSH(X,\theta)$ is called a {\em model potential} if $\int_X\theta_\phi^n>0$ and $\phi=P_\theta[\phi]$.
\end{definition}

{Recall that the Monge-Amp\`ere capacity of a Borel set $E\subset X$ is defined as
\[\Capa_\theta(E)=\sup\left\{\int_E \theta_u^n: u\in \PSH(X, \theta), V_\theta-1\leq u\leq V_\theta\right\}.\]
A function is called \emph{quasi-continuous} if for each $\varepsilon>0$, there exists an
open set $U$ such that $\Capa_\theta(U)<\varepsilon$ and the restriction of $u$ on 
$X\setminus U$ is continuous. 
} 

\begin{lemma}\label{lem: quasi-everywhere} Assume that $f$ is bounded from below.
Then
 \[P_\theta(f)=\sup\{u\in\PSH(X,\theta): u\leq f\;\text{ quasi-everywhere in}\, X \}. \]
 \end{lemma}

 \begin{proof} See \cite[Lemma 2.3]{DDP25-singularities}.
 \end{proof}

\begin{theorem} \label{thm: envelope}
	Assume that $f$ is quasi-continuous, bounded from below, and $P_\theta(f)\in\PSH(X,\theta)$. Then $\theta_{P_\theta(f)}^n$ is concentrated
	on the contact set $\{P_\theta(f)=f \}$.
\end{theorem}
\begin{proof}
 The proof proceeds along the same lines as in \cite[Theorem 2.2]{darvas2025relative}, where Lemma~\ref{lem: quasi-everywhere} is used. For this reason, we omit the details.
\end{proof}
\begin{lemma}[{\cite[Lemma 5.11]{darvas2025relative}}]\label{lemma: ddl511}
     Let $u\in\mathcal{E}(X,\theta,\phi)$ be such that $u\leq \phi$. Let $\gamma:\mathbb{R}^+\cup\{+\infty\}\rightarrow\mathbb{R}^+\cup\{+\infty\}$ be an increasing concave continuous function with $\gamma'\leq 1$. Then $v:=-\gamma(\phi-u)+\phi\in\PSH(X,\theta,\phi)$ and 
     \[ \theta_v^n\geq (\gamma'(\phi-u))^n\theta_u^n.\]
\end{lemma}

   \begin{lemma} \label{lem: envelope}
       Let $\varphi\in\mathcal{E}(X,\theta,\phi)$ be such that $\varphi\leq \phi$. Let $\chi:\mathbb{R}^+\rightarrow\mathbb{R}^+$ be an increasing convex continuous function with $\chi'\geq 1$. If $\psi=-\chi(\phi-\varphi)+\phi$ then
       \[(\theta+\dc P_\theta(\psi))^n\leq \mathbf{1}_{\{P_\theta(\psi)=\psi\}}(\chi'(\phi-\varphi))^n(\theta+\dc\varphi)^n. \]
   \end{lemma} 
   \begin{proof} Set $u=P_\theta(\psi)$. 
       Let $\gamma:\mathbb{R}^+\rightarrow\mathbb{R}^+$ be the inverse function of $\chi$. Then, $\gamma$ is concave, increasing and $\gamma'\leq 1$. We observe that
       \[v:=-\gamma(\phi-u)+\phi\leq -\gamma(\phi-\psi)+\phi= -\gamma(\chi(\phi-\varphi))+\phi=\varphi. \]
       We apply {Lemma \ref{lemma: ddl511} and \ref{lem: maxprin}} to obtain
       \[\mathbf{1}_{\{ u=\psi\}}(\chi'(\phi-\varphi))^{-n}\theta_u^n= \mathbf{1}_{\{ u=\psi\}}(\gamma'(\phi-u))^n\theta_u^n\leq  \mathbf{1}_{\{ u=\psi\}}\theta_v^n\leq \mathbf{1}_{\{ u=\psi\}}\theta_\varphi^n,\] since $\{u=\psi\}=\{v=\varphi\}$.
       Together with Theorem~\ref{thm: envelope}, we have
       \[\theta_u^n=\mathbf{1}_{\{ u=\psi\}}\theta_u^n\leq \mathbf{1}_{\{ u=\psi\}} (\chi'(\phi-\varphi))^n\theta_\varphi^n.\]
   \end{proof}
   \begin{lemma}[{\cite[Lemma 3.4]{darvas2025relative}}]\label{lem: lem34DDL25}
       Let $\phi\in\PSH(X,\theta)$ be a model potential. If $b>1$ and
$u,v\in\mathcal{E}(X,\theta,\phi)$ then $P_\theta(bu-(b-1)v)\in\mathcal{E}(X,\theta,\phi)$.
   \end{lemma}
   \begin{proposition}[Domination principle] \label{prop: domination}
       Let $u,v$ be $\theta$-psh functions such that $u\in\mathcal{E}(X,\theta,\phi)$. If 
       \[\theta_u^n(\{u<v \})\leq c\theta_v^n(\{u<v\}) \] for some $c\in[0,1)$, then $u\geq v$.
   \end{proposition}
   The case $c=0$ is treated in ~\cite[Proposition 3.11]{darvas2018monotonicity} (see also \cite[Theorem 3.4]{darvas2025relative}), while the general case follows the arguments in~\cite[Proposition 1.12]{guedj2021quasi}. For the sake of completeness, we provide the proof here.
   \begin{proof}

Since $$c(\theta+\dc\max(u,v))^n\geq \theta_u^n$$ on $\{u<v\}$,
       we may assume that $u\leq v$. 

       Fix $b>1$ and set $u_b:=P_\theta(bu-(b-1)v)$. Since the masses of $u$ and $v$ are equal, it follows from Lemma~\ref{lem: lem34DDL25} that $u_b\in\mathcal{E}(X,\theta,\phi)$. Set $D_b=\{u_b=bu-(b-1)v\}$. By Theorem~\ref{thm: envelope}, $\theta_{u_b}^n$ is concentrated on the contact set $D_b$ and $u_b\leq bu-(b-1)v$.  Hence, $b^{-1}u_b+(1-b^{-1})v\leq u$. It follows from Lemma~\ref{lem: maxprin} that
       \begin{align*}
           \mathbf{1}_{D_b} b^{-n}\theta_{u_b}^n+(1-b^{-1})^n \mathbf{1}_{D_b}\theta_v^n\leq \mathbf{1}_{D_b}(\theta+\dc(b^{-1}u_b+(1-b^{-1})v))^n \leq\mathbf{1}_{D_b}\theta_u^n.
       \end{align*}
  Hence, we get by the hypothesis that
   \begin{align*}
           \mathbf{1}_{D_b\cap\{u<v\}} b^{-n}\theta_{u_b}^n+(1-b^{-1})^n \mathbf{1}_{D_b\cap\{u<v\}}\theta_v^n \leq\mathbf{1}_{D_b\cap\{u<v\}}\theta_u^n.
       \end{align*}
       We take $b>1$ so large that $(1-b^{-1})^n>c$, so $\theta_{u_b}^n$ is concentrated on $D_b\cap\{u=v\}$. Since $u_b\leq u$ with the equality on $D_b\cap\{u=v\}=\{u_b=u\}$, Lemma~\ref{lem: maxprin} again yields
       \[\theta_{u_b}^n=  \mathbf{1}_{D_b\cap\{u_b=v\}}\theta^n_{u_b}=\mathbf{1}_{\{u_b=u\}}\theta^n_{u_b}\leq  \mathbf{1}_{\{u_b=u\}}\theta^n_{u}\leq \theta_u^n.\]
       Comparing the total mass, we obtain $\theta_{u_b}^n=\theta_u^n$. Hence, we have
       \[\int_X e^{u_b}\theta_u^n=\int_X e^{u_b}\theta^n_{u_b}=\int_{\{u_b=u\}} e^{u_b}\theta^n_{u_b}=\int_{\{u_b=u\}} e^{u}\theta^n_{u}=\int_X e^u\theta_u^n>0.\]
      We note that since $u\leq v$, $u_b$ is decreasing in $b$. Thus, the above implies that $\varphi=\lim_{b\to+\infty}u_b$ is not identically $-\infty$. For $x\in\{u<v\}$, we have
      \[u_b(x)\leq b(u(x)-v(x))+v(x).\] Letting $b\to\infty$, we
see that $\varphi(x)=-\infty$. This implies that $\{u<v\}$ is pluripolar and therefore empty. Therefore, we conclude that $u=v$.
   \end{proof}
   \begin{proposition}[{\cite[Proposition 4.24]{darvas2018monotonicity}}]\label{prop: unique}
  Let $u,v\in\mathcal{E}(X,\theta,\phi)$. If $e^{-\lambda v}\theta_v^n\geq e^{-\lambda u}\theta_u^n$ for some $\lambda>0$, then $u\geq v$.    
   \end{proposition}
   \begin{proof} Fix $a>0$.
       It follows from Proposition~\ref{prop: domination} that in $\{u<v-a\}$, we have 
       \[ \theta_u^n\leq e^{\lambda(u-v)}\theta_v^n\leq e^{-\lambda a}\theta_{v-a}^n,\] so $u\geq v-a$. Letting $a\to 0$, the conclusion follows.
   \end{proof}
   
\subsection{Energy class and Entropy}

We say that a {\em weight }is a continuous strictly increasing function $\chi:[0,+\infty)\rightarrow[0,+\infty)$ such that $\chi(0)=0$ and $\chi(+\infty)=+\infty$.

We fix $\phi$ a model potential and let $\mathcal{E}_\chi(X,\theta,\phi)$ denote the set of all $\varphi\in\mathcal{E}(X,\theta,\phi)$ such
that \[E_\chi(\varphi,\phi):=\int_X\chi(\phi-\varphi)\MA_\theta(\varphi)<+\infty. \]
In the special case $\chi(t)=t^p$, $p>0$, we simply denote the relative energy class
with $\mathcal{E}^p(X,\theta,\phi)$ and the corresponding relative energy $E_p(\varphi,\phi)$.  When $\phi=V_\theta$, we simply write $\mathcal{E}^p(X,\theta)$ and $E_p(\varphi)$.

\medskip

We recall that given two positive probability measures $\mu$, $\nu$, the relative entropy $Ent_p(\mu,\nu)$ is
defined as
\begin{equation}
    {\rm Ent}_p(\mu,\nu)=\int_X\left|\log\left( \frac{d\mu}{d\nu}\right) \right|^p d\nu,
\end{equation}
if $\mu$ is absolutely continuous with respect to $\nu$, and $+\infty$ otherwise.
When $\nu=\omega_X^n$ and $\mu=f\omega_X^n$, we simply denote the $p$-Nash-Yau entropy by \begin{equation}\label{eq: entropy}
    {\rm Ent}_p(f):=\int_X f|\log f|^p\omega_X^n.
\end{equation}

\subsection{Orlicz space}
We recall some background about the Orlicz space; see e.g.~\cite{rao2002applications} for more details. 
\begin{definition}
    A {\em Young} function is a convex increasing lower semicontinuous function $w:\mathbb{R}^+\rightarrow\mathbb{R}^+$ such that $w(0)=0$, $\lim_{t\to 0^+}\frac{w(t)}{t}=0$ and $\lim_{t\to+\infty}\frac{w(t)}{t}=+\infty$. Its {conjugate function} is the Lengendre transform of $w(\cdot)$, i.e.,
    \[ w^*(t)=\sup_{s\geq 0}(st-\chi(s)).\]
\end{definition}

\begin{definition}\label{def: Orlicz-norm}
    Let $\mu$ be a positive measure on $X$ and let $w$ be a Young function. The {\em Orlicz space} $L^w(X,d\mu)$ is the set of all measurable functions $f$ on $X$ such that
    \[\int_X w(\varepsilon f)d\mu<+\infty \]
    for some $\varepsilon>0$. The {\em Luxembourg norm} of $f\in L^w(\mu)$ is 
    \[ \|f\|_{L^w(\mu)}:=\inf\left\{r>0:\int_X w(r^{-1}f)d\mu\leq 1 \right\}.\]
\end{definition}
When $w(t)=t^p$ for $p\geq 1$ and $\mu=dV$ is a smooth volume form on $X$, we simply write $\|f\|_p$.

In honor of the breakthrough result of S. Ko\l odziej ~\cite{kolodziej1998complex}, we introduce the following notion.
\begin{definition} \label{conditionK}
   \hypertarget{conditionK}{We say that a Young function $w$ satisfies {\em condition (K)} if there exists an increasing function $h:(0,+\infty)\to(1,+\infty)$ such that
    \[w(t)=t h(\log(1+t)), \qquad\int_1^{+\infty}h^{-1/n}(t) dt<+\infty.\]}
\end{definition}
\begin{example} We have several examples of weights satisfying condition (K):
 \begin{itemize}
        \item $w(t)=t^p$, for some $p>1$;
        \item $w(t)=t(\log(1+t))^{p}$, for some $p>n$;
        \item $w(t)=t(\log(1+t))^n(\log(1+\log(1+t)))^p$, for some $p>n$.
    \end{itemize}
\end{example}
We have the following classical lemma. 
\begin{lemma}\label{lem: Young-fuction} Let $\mu$ be a positive measure on a measured space $X$.
    Let $f:X\rightarrow\mathbb{R}$ be a measurable function such that $\int_Xf d\mu<+\infty$. For all $C>1$, there exists a Young function $w$ such that 
    \begin{equation}
        \int_X w(f)d\mu\leq C\int_X fd\mu.
    \end{equation}
\end{lemma}
\begin{proof}
    By Fubini's theorem, we have
    \[\int_Xw(f)d\mu=\int^{+\infty}_0 w'(t)\mu(f>t)dt. \]
    {For any $C>1$,} let $q\in(0,1)$ be such that $C(1-q)=1$. Set $$I=\int_X fd\mu=\int_0^{+\infty}\mu(f>t)dt<+\infty.$$
    Since $\int_N^{+\infty}\mu(f>t)dt=0$ as $N\to+\infty$ we can find a sequence $N_k$ such that $N_{k+1}\geq 2N_k$ and 
    {\[\int_{N_k}^{+\infty}\mu(f>t)dt<q^{2k}I.\]} 
    Defining $w'(t)=q^{-k}$ if $N_k\leq t<N_{k+1}$, we obtain the desired inequality.
\end{proof}


   \section{A priori estimates} \label{sect: estimate}
   In this section, we let $dV$ always denote the Lebesgue measure on $X$. Let $\theta$ be a closed smooth (1,1) form on $X$ representing a big cohomology class. Assume $\theta\leq A\omega_X$ for a fixed $A>0$. 
   For simplicity, we denote by $\|f\|_p$ the $L^p$-norm of the space $L^p(X,dV)$. 
   
   \subsection{Fixed reference form} 
Let $\mu$ be a probability measure on $X$. In~\cite{boucksom2010monge,berman2013variational,darvas2021log} it is shown that there exists a unique $\varphi\in\mathcal{E}(X,\theta)$ such that 
\begin{equation}
    \label{eq: cmae-big} \MA_\theta(\varphi)=\mu,\quad \sup_X\varphi=-1.
\end{equation}

Our main theorem in this section is as follows.
\begin{theorem}\label{thm: main}
    Assume that $\mu= fdV$ is a probability measure on $X$, and $0\leq f\in L^w(dV)$ and $f>0$ almost everywhere, where $w$ satisfies \hyperlink{conditionK}{condition (K)}. Then there exists a unique solution $\varphi\in\mathcal{E}(X,\theta)$ to~\eqref{eq: cmae-big} with minimal singularities. More precisely, there is a constant $C$ depending on $A$, $\omega_X$, $dV$, $\theta$, $n$, $w$, and $\|f\|_w$ such that 
    \[ \varphi\geq V_\theta-C.\]
\end{theorem}

In the sequel, $C$ denotes a uniform constant that may change from line to line. Unless explicitly indicated otherwise, $C$ depends only on $A$, $\omega_X$, $dV$, $\theta $, $n$, $w$, $\|f\|_w$ unless otherwise specified.

We set the sub-level set $\Omega_s:
=\{\varphi<V_\theta-s\}$ for $s>0$. 
Fix $r>0$. 
We solve an auxiliary complex Monge--Amp\`ere equation 
\begin{equation}
    \MA_\theta(v_s)=\frac{\mathbf{1}_{\Omega_s}(V_\theta-\varphi-s)^r }{A_s}fdV,\; v_s\in\mathcal{E}(X,\theta),\quad\sup_X v_s=-1,
\end{equation} where $A_s=\int_{\Omega_s}(V_\theta-\varphi-s)^rfdV$. The existence of solution $v_s$ is shown in~\cite{boucksom2010monge,berman2013variational,darvas2021log}.
We compare the auxiliary function with the solution.
\begin{lemma}\label{lem: key}
With assumptions in Theorem~\ref{thm: main}, we have
    \begin{equation}\label{eq: key-equation}
       \frac{n}{n+r} (2A_s)^{-\frac{1}{n}}(V_\theta-\varphi-s)^{\frac{n+r}{n}}\leq V_\theta-v_s+\frac{n}{n+r}(2A_s)^{\frac{1}{r}}.
    \end{equation}
\end{lemma}
\begin{proof}

    Set $$\varepsilon=\left(\frac{n+r}{n}\right)^{\frac{n}{n+r}}(2A_s)^{\frac{1}{n+r}},\quad\Lambda=\frac{n}{n+r}(2A_s)^{\frac{1}{r}},\quad C_\Lambda=\varepsilon\Lambda^{\frac{n}{n+r}}$$ so that $\frac{\varepsilon n}{n+r}\Lambda^{-\frac{r}{n+r}}=1$. 
    If we set $\chi(t)=\varepsilon^{-\frac{n+r}{n}}t^{\frac{n+r}{n}}-\Lambda$, then $\chi$ is a convex increasing function on $(C_\Lambda,+\infty)$. 
    Let $\gamma$ be the inverse function of $\chi$, which is concave and increasing. By computations, we have
    \[ \gamma(t)=\varepsilon(t+\Lambda)^{\frac{n}{n+r}}\] and $\gamma'(0)=1$ so that $\gamma'\leq 1$.

Set $\widetilde\Omega_s:=\{C_\Lambda<V_\theta-\varphi-s\}$.
We see that $\widetilde\Omega_s\subset\Omega_s$. We set $$\psi:= V_\theta-\chi(\mathbf{1}_{\widetilde\Omega_s}(V_\theta-\varphi-s)) \; \text{and}\; u:=P_\theta(\psi).$$ We observe that $u$ has minimal singularities, so $u\in\mathcal{E}(X,\theta)$. We denote by $\mathcal{C}=\{u=\psi\}$ the contact set. We observe that 
   \[v:=-\gamma(V_\theta-u)+V_\theta\leq-\gamma(V_\theta-\psi)+V_\theta= -\mathbf{1}_{\widetilde\Omega_s}(V_\theta-\varphi-s)+V_\theta\leq \varphi+s, \]
    with equality on the contact set $\mathcal{C}=\{u=\psi\}$.
Since $v_s\leq V_\theta$ we have
\[\{u<v_s\}\cap\mathcal{C}\subset \{C_\Lambda<V_\theta-\varphi-s\}\cap\mathcal{C}. \]  
    By Lemmas~\ref{lemma: ddl511} and~\ref{lem: maxprin}, on the set $\{u<v_s\}$ we have
    \begin{equation}
        \begin{split}
        (\chi'(\mathbf{1}_{\Omega_s}(V_\theta-\varphi-s)))^{-n}\theta_u^n&=    \mathbf{1}_{\mathcal{C}}(\chi'(V_\theta-\varphi-s))^{-n}\theta_u^n\\
        &= \mathbf{1}_{\mathcal{C}}(\gamma'(V_\theta-u))^n\theta_u^n\\
        &\leq  \mathbf{1}_{\mathcal{C}}\theta_v^n\leq \mathbf{1}_{\mathcal{C}}\theta_\varphi^n,
        \end{split}
    \end{equation} since we also have $\mathcal{C}=\{v=\varphi+s\}$. Thus, on $\{u<v_s\}$,
    \begin{align*}
        \MA_\theta(u)&\leq (2A_s)^{-1}\mathbf{1}_{\widetilde\Omega_s}(V_\theta-\varphi-s)^{r}\MA_\theta(\varphi)\\
        &\leq (2A_s)^{-1}\mathbf{1}_{\Omega_s}(V_\theta-\varphi-s)^{r}fdV \\
        &= \frac{1}{2}\MA_\theta(v_s).
    \end{align*}
 We apply the domination principle (Proposition~\ref{prop: domination}) to obtain
    $v_s\leq u $, so $v_s\leq \psi$. This means that
    \[v_s\leq V_\theta-  \frac{n}{n+r} (2A_s)^{-\frac{1}{n}}\mathbf{1}_{\widetilde\Omega_s}(V_\theta-\varphi-s)^{\frac{n+r}{n}}+\left(\frac{n}{n+r}\right)(2A_s)^{\frac{1}{r}}.\]
On the other hand, on $\Omega_s\setminus\widetilde\Omega_s$, we have $0<V_\theta-\varphi-s\leq C_\Lambda=(2A_s)^{\frac{1}{r}}$, so
\[\frac{n}{n+r} (2A_s)^{-\frac{1}{n}}\mathbf{1}_{\Omega_s\setminus\widetilde\Omega_s}(V_\theta-\varphi-s)^{\frac{n+r}{n}}\leq \frac{n}{n+r}(2A_s)^{\frac{1}{r}}. \]
Since  $-V_\theta\leq -v_s$, we obtain the inequality~\eqref{eq: key-equation}. 
Since on $X\setminus\Omega_s$, the inequality is trivial and the proof is complete.
\end{proof}

\subsubsection{Density $L^{1+\varepsilon}$}

We first prove Theorem~\ref{thm: main} in the case when $w(t)=t^p$, for some $p>1$, i.e., $f\in L^{p}(X,dV)$. 
\begin{theorem}\label{thm: priori-Lp}
    Let $\varphi$ be a $\theta$-psh function with minimal singularities. Moreover, assume that $\sup_X\varphi=0$ and $\MA_\theta(\varphi)\leq fdV$ for $0\leq f\in L^p(X,dV)$, with $p>1$ and $f>0$ almost everywhere. Then there exists a constant $C$ only depending on $dV$, $A$, $\omega_X$, $n$, $p$ and $\|f\|_{L^p}$ 
    such that \[ \varphi\geq V_\theta-C.\]
\end{theorem}

\begin{proof} Assume that $\theta\leq A\omega_X$ for $A>1$. Let $p^*$ be the conjugate exponent of $p$, i.e., $\frac{1}{p}+\frac{1}{p^*}=1$. We apply Lemma~\ref{lem: key} with $r=\frac{1}{p^*}$.
By the H\"older inequality, we have
\begin{equation*}
    A_s\leq \int_X(V_\theta-\varphi)_+^{{1}/{p^*}}fdV\leq \|f\|_{p}(\|V_\theta\|_{1}+\|\varphi\|_{1})^{1/p^*}.
\end{equation*} 
Since $\varphi$ and $V_\theta$ both belong to the compact set of $A\omega_X$-psh functions normalized by $\sup_X u=0$, their $L^{1}(dV)$ norms are bounded by an absolute constant only depending on $dV$, $\theta$, $A,\omega_X$, by Theorem~\ref{thm: Skoda/Tian}. Hence $\Lambda{=\frac{n}{n+r}(2A_s)^{\frac{1}{r}}}\leq C\|f\|_p^{p^*
}$. 



Define $\phi(s):=\int_{\Omega_s}fdV$. Fix $q>1$ to be chosen later. 
 Lemma~\ref{lem: key} (with $r=\frac{1}{p^*}$) and H\"older's inequality yield
\begin{align*}
    A_s&\leq \int_{\Omega_s}(2A_s)^{\frac{r}{n+r}}(V_\theta-v_s+\Lambda)^{\frac{n}{n+r}}fdV\\
    &\leq 2^{\frac{r}{n+r}}A_s^{\frac{r}{n+r}}\left(\int_{\Omega_s}(V_\theta-v_s+\Lambda)^{\frac{nq}{n+r}}fdV \right)^{1/q}\left(\int_{\Omega_s}fdV \right)^{1-1/q}\\
    &\leq 2^{\frac{r}{n+r}}A_s^{\frac{r}{n+r}}\|f\|_{p}^{1/q}\|(V_\theta-v_s+\Lambda)\|_{\frac{nqp^*}{n+r}}^{1/q}\phi(s)^{1-\frac{1}{q}}\\
    &\leq 2^{\frac{r}{n+r}}A_s^{\frac{r}{n+r}}\|f\|_{p}^{1/q} (\|v_s\|_{\frac{nqp^*}{n+r}}+\Lambda)^{\frac{n}{n+r}}\phi(s)^{1-\frac{1}{q}}\\
     &\leq C(n,p,q,\|f\|_p)A_s^{\frac{r}{n+r}}\phi(s)^{1-\frac{1}{q}}.
\end{align*} Hence $A_s\leq C_0\phi(s)^{(1-\frac{1}{q})\frac{n+r}{n}}$.
On the other hand, for any $t\in [0,1]$, we have
\[A_s =\int_{\Omega_s}(V_\theta-\varphi-s)_+fdV\geq t\phi(s+t).\] It follows that 
\[t\phi(s+t)\leq C_0\phi(s)^{1+\delta_0},\]
where we choose $q>\frac{n}{r}+1=np^*+1$ so that $\delta_0=(1-\frac{1}{q})\frac{n+r}{n}-1\in(0,\frac{1}{np^*})$. 
We observe that 
\[ \phi(s)\leq \frac{1}{s}\int_{\Omega_s}(V_\theta-\varphi)^{\frac{1}{p^*}}fdV\leq \frac{C_1}{s}\xrightarrow{s\to\infty}0\] so there exists $s_0>0$ such that $\phi(s_0)^{\delta_0}<\frac{1}{2C_0}$.
Since $\phi$ is non-increasing and right-continuous we can apply De Giorgi's iteration lemma (see, e.g., ~\cite[Lemma 2.4]{eyssidieux2009singular}) to obtain that $\phi(s)=0$ for all $s\geq S_\infty=s_0+\frac{2C_0\phi(s_0)^{\delta_0}}{1-2^{-\delta_0}}$. 

Therefore, we infer that $\theta_\varphi^n(\varphi< V_\theta-S_\infty)=0$. It follows from the domination principle (Proposition~\ref{prop: domination}) that $\varphi\geq V_\theta-S_\infty$ hold everywhere.
\end{proof}
\begin{remark}
    We can follow the same arguments as in ~\cite[pages 400-403]{GuoPhongTong23-estimates} (with $r=1$ in Lemma~\ref{lem: key}) to get the sightly generalized result when the density $f$ is just assumed to have the $L^1(\log L)^p$- bound for $p>n$. In the aforementioned paper, instead, we use the inequality H\"older-Young, and $A_s$ is bounded from above by the energy of $\varphi$. And, the energy can also be controlled in the case of complex Monge--Amp\`ere equations.
\end{remark}

We can provide the proof of Theorem~\ref{thm: main} when the density belongs to $L^{1+\varepsilon}$. It was treated in ~\cite{boucksom2010monge,darvas2018monotonicity}.
\begin{theorem}\label{thm: main-Lp}
      Assume that $\mu= fdV$ is a probability measure on $X$ and $f\in L^p(dV)$ for some $p>1$, and $f>0$ almost everywhere. Then there exists a unique solution $\varphi\in\mathcal{E}(X,\theta)$ to~\eqref{eq: cmae-big} with minimal singularities. More precisely, there is a constant $C$ depending on $dV$, $n$, $A$, $\omega_X$, $p$ and $\|f\|_p$ such that 
    \[ \varphi\geq V_\theta-C.\]
\end{theorem}
\begin{proof}
    The arguments follow from~\cite[p. 1099]{darvas2025relative}.
     In~\cite[Theorem A]{boucksom2010monge} (see also~\cite[Theorem 4.7]{darvas2021log}) it is shown that there exists $\varphi\in\mathcal{E}(X,\theta)$ such that $\MA_\theta(\varphi)=\mu$. It suffices to prove that $\varphi$ has minimal singularities, i.e., that $\varphi-V_\theta$ is bounded. 

    By Skoda's uniform
integrability theorem (Theorem~\ref{thm: Skoda/Tian}), we can choose $\varepsilon>0$ and $1<q<p$ such that $e^{-\varepsilon u}f\in L^q$ for all $u\in\PSH(X,A\omega_X)$, and $\sup_X u=0$. For each $j\in\mathbb N$ we solve
    \[\varphi_j\in\mathcal{E}(X,\theta),\quad \MA_\theta(\varphi_j)=\mathbf{1}_{\{\varphi>V_\theta-j\}}e^{\varepsilon(\varphi_j-\max(\varphi,V_\theta-j))}fdV. \]
We see that for each $j$, $\max(\varphi,V_\theta-j)$ is a subsolution of the above equation, because, by the plurifine property (Lemma~\ref{lem: plurifine})
\begin{align*}
    \MA_\theta({\max(\varphi,V_\theta-j)})&\geq\mathbf{1}_{\{\varphi>V_\theta-j\}}\MA_\theta({\max(\varphi,V_\theta-j)})\\
    &= \mathbf{1}_{\{\varphi>V_\theta-j\}}\MA_\theta(\varphi)\\
    &=\mathbf{1}_{\{\varphi>V_\theta-j\}} e^{\varepsilon(\max(\varphi,V_\theta-j)-\max(\varphi,V_\theta-j))} fdV.
\end{align*}
Hence, $\varphi_j\geq \max(\varphi,V_\theta-j)$ by Proposition~\ref{prop: unique}, i.e., $\varphi_j-V_\theta$ is bounded. 
Again, we have $\varphi_j\geq \varphi_k$ for $j<k$. We see that the densities 
{\[f_j\leq e^{\varepsilon\sup_X\varphi_1}e^{-\varepsilon\max(\varphi, V_\theta-j)}f\]}are uniformly bounded in $L^q(dV)$. By uniform a priori estimate (Theorem~\ref{thm: priori-Lp}), we have 
    $\varphi_j-V_\theta\geq C$ for $C>0$ only depending on $A$, $n$, $\omega_X$, $dV$, $p$ and $\|f\|_p$. Thus, $\psi:=\lim_j\varphi_j\geq V_\theta-C$ and $\MA_\theta(\psi)=e^{\varepsilon(\psi-\varphi)}fdV$, so $\varphi=\psi$ by Proposition~\ref{prop: unique}.
\end{proof}

\subsubsection{General case} We establish the following estimate. 
\begin{proposition}\label{prop: moreestimate} Let $\psi\in\PSH(X,\theta)$ such that $\sup_X\psi=0$, $a\in (0,1)$ and let $0\leq  f\in L^p(dV)$ for some $p>1$.
    Let $\varphi\in\mathcal{E}(X,\theta)$ be such that $\sup_X\varphi=0$ and 
    \[\theta_\varphi^n\leq a^n\theta_\psi^n +fdV.\]
    Then there exists a constant $C>0$ depending on $dV$, $A$, $\omega_X$, $n$, $p$, $a$ and $\|f\|_p$ such that 
    \begin{equation*}
        \varphi\geq a\psi+(1-a)V_\theta-C.
    \end{equation*}
\end{proposition}
The proposition is a slight generalization of Ko\l odziej’s $L^\infty$ estimate, as shown in~\cite[Theorem 3.3]{darvas2021log} where the relative Monge-Amp\`ere capacity is applied. We provide an alternative proof, using the quasi-psh envelopes. The arguments are inspired by Lu and Nguyen~\cite[Lemma 5.3]{LuNguyen22-hessian}. 
\begin{proof} 
Fix $b=(1-a)^{-1}>1$.
It follows from~\cite[Theorem 3.3, Lemma 3.3]{darvas2025relative} that $P_\theta(b\varphi-(b-1)V_\theta)\in\mathcal{E}(X,\theta)$. Set $u_b:=P_\theta(b\varphi-(b-1)\psi)$. Since $$P_\theta(b\varphi-(b-1)\psi)\geq P_\theta(b\varphi-(b-1)V_\theta),$$ we have $u_b\in\mathcal{E}(X,\theta)$. 
Denote the contact set $\mathcal{C}=\{u_b=b\varphi-(b-1)\psi\}$. We observe that 
\[\varphi_b:=b^{-1}u_b+(1-b^{-1})\psi\leq \varphi \] with equality on $\mathcal{C}$. Applying the maximum principle (Lemma~\ref{lem: maxprin}) we have $\mathbf{1}_\mathcal{C}\theta^n_{\varphi_b}\leq \mathbf{1}_\mathcal{C}\theta^n_\varphi$. It follows from Theorem~\ref{thm: envelope} that $\theta_{u_b}^n$ is concentrated on the contact set $\mathcal{C}$. Therefore, we obtain
\begin{equation*}
 \mathbf{1}_\mathcal{C}(b^{-n}\theta^n_{u_b}+(1-b^{-1})^{n}\theta_\psi^n)\leq \mathbf{1}_\mathcal{C}\theta_{\varphi_b}^n\leq\mathbf{1}_{\mathcal{C}}\theta_\varphi^n,  
\end{equation*} this implies $\theta^n_{u_b}\leq b^nfdV$ by assumptions. Next, we want to bound $\sup_X u_b$. Theorem~\ref{thm: envelope} yields
\begin{align*}
    \Vol(\theta)(-\sup_X u_b)^{1/p^*}\leq\int_X |u_b|^{1/p^*}\theta_{u_b}^n&=\int_{\mathcal{C}}|b\varphi-(b-1)\psi|^{1/p^*}\theta_{u_b}^n\\
    &\leq \int_{\mathcal{C}}|b\varphi-(b-1)\psi|^{1/p^*}b^nfdV\\
    &\leq b^n \|f\|_p (b\|\varphi\|_1+b\|\psi\|_1)^{1/p^*},
\end{align*} as follows from
 H\"older's inequality. We see that the latter term is uniformly bounded since $\sup_X\varphi=\sup_X\psi=0$; cf. Theorem~\ref{thm: Skoda/Tian}. This shows that $\sup_Xu_b$ is uniformly bounded from below. We can apply Theorem~\ref{thm: main-Lp} to obtain $u_b\geq V_\theta-C$ for $C>0$ depending on $A$, $\omega_X$, $dV$, $n$, $p$ and $\|f\|_{p}$. Since $u_b\leq b\varphi-(b-1)\psi$ and $a=1-b^{-1}$ we have
\[\varphi\geq a\psi+(1-a)V_\theta-C(1-a)^{-1}.  \]
\end{proof}

\begin{remark} We can adapt the arguments in~\cite{GuedjLu25-estimate} to provide an alternative proof.
    By Skoda's uniform integrability theorem and H\"older's inequality, we can find $\varepsilon>0$ and $1<q<p$ such that $e^{-\varepsilon u}f\in L^q(dV)$ for all $u\in\PSH(X,\omega_X)$, $\sup_Xu=0$.
    Fix $\delta\in(0,1)$ such that $\lambda=a/\delta\in(0,1)$.
    Set $h=e^{-\varepsilon\varphi}f\in L^q(dV)$ for $1\in (1,p)$. Let $v\in\mathcal{E}(X,\theta)$ solve \[\MA_\theta(v)=\frac{1}{\Vol(\theta)}(\theta+\dc v)^n=\frac{hdV}{\int_XhdV},\quad\sup_Xv=0. \]
    By Step 1 in Theorem~\ref{thm: main}, $v\geq V_\theta-C$ where $C_0$ depends only on $\theta$, $dV$, $n$, $p$, $\varepsilon$ and $\|f\|_p$.  We set $u:=\delta\psi+(1-\delta)v-C$ for $C>0$ to be chosen hereafter. We observe that
    \[\theta_u^n\geq \delta^n\theta_\psi^n +(1-\delta)^n\theta_v^n.\]
    Since $\int_XhdV\leq \|h\|_q\leq C_1\|f\|_p$ and $\psi,v\leq 0$, we have
    \[\MA_\theta(v)\geq  \frac{1}{C_1\|f\|_p}e^{-\varepsilon\varphi}fdV\geq e^{\varepsilon C-\log[C_1\|f\|_p]}e^{-\varepsilon(\varphi-u)}fdV.\]
    If we choose $$C=\varepsilon^{-1}\log[C_1\|f\|_p\Vol(\theta)^{-1} a^{-n}\delta^{n}(1-\delta)^{-n}],$$ then we obtain
     \begin{align*}
    \lambda^n\theta_u^n&\geq  \lambda^n\delta^n\theta_\psi^n+\lambda^n(1-\delta)^n  \Vol(\theta)  e^{\varepsilon C-\log[C_1\|f\|_p]}e^{-\varepsilon(\varphi-u)}fdV\\
    &=a^n\theta_\psi^n+e^{-\varepsilon(\varphi-u)}fdV.
    \end{align*}
   Thus, on the set $\{\varphi<u\}$ we have $\lambda^n\theta_u^n\geq \theta_\varphi^n$. Since $\lambda\in[0,1)$, by the domination principle, we have $\varphi\geq u$. Therefore, for any $\delta\in (a,1)$,
   \[\varphi\geq \delta\psi+(1-\delta)V_\theta-C. \]
\end{remark}

In order to prove Theorem~\ref{thm: main} in the general case, we follow the trick in~\cite[Theorem 2.1]{GuedjLu25-estimate} and in~\cite[Theorem 3.1]{liu2024relative} to reduce to the case of densities in $L^{1+\varepsilon}$. 
\begin{lemma} \label{lem: reduceLp}
Assume that $\mu=fdV$ is a probability measure on $X$ and $f\in L^w(dV)$, where $w:[0+\infty)\to[0+\infty)$ is a convex increasing function. Let $\varphi\in\mathcal{E}(X,\theta)$ be a solution to~\eqref{eq: cmae-big}. Then, for any $a\in(0,1)$, $p>1$ there exist a function $\psi\in\PSH(X,\theta)$ and a density $g\in L^p(dV)$ such that
\[\MA_\theta(\varphi)\leq a^n\MA_\theta(\psi)+ gdV. \]
    
\end{lemma}
\begin{proof}It is shown in~\cite[Theorem A]{boucksom2010monge} that there exists $v\in\mathcal{E}(X,\theta)$ which solves
\begin{equation}
    \MA_\theta(v)=cw\circ f dV,\quad\sup_Xv=0\quad\text{where}\; c^{-1}=\int_Xw\circ f dV.
\end{equation}
Thanks to Tian's $\alpha$-invariant or Skoda's integrability theorem (Theorem~\ref{eq: alpha}), 
    we can choose $\alpha=\alpha(n,A\omega_X)/p$ so that $e^{-\alpha v}\in L^p$ for $p>1$. Fix $B>0$ large enough to be chosen later. 
   We let $\gamma:\mathbb{R}^+\rightarrow\mathbb{R}^+$ be a smooth concave increasing function such that $0\leq \gamma'\leq 1$. We set \begin{equation}\label{eq: psi}
       \psi:=V_\theta-\gamma(V_\theta-v),
   \end{equation} which is $\theta$-psh.
   It follows from Lemma~\ref{lemma: ddl511} that $$\theta^n_\psi\geq (\gamma'(V_\theta-v))^n\theta_v^n.$$
On $\{\log f<-\alpha v+B\}$, we have
\[\MA_\theta(\varphi)\leq e^B e^{-\alpha v}dV. \] On the other hand, on $\{\log f>-\alpha v+B\}\subset\{\log f>\alpha V_\theta-\alpha v+B\}$, 
\[\MA_\theta(\psi)\geq (\gamma'(V_\theta-v))^n\MA_\theta(v)\geq \left[\gamma'\left(\alpha^{-1}(\log f-B)\right)\right]^n cw(f)dV. \]
If we choose $\gamma$ so that for all $s>B$, then
\[c\left[\gamma'\left(\alpha^{-1}(s-B)\right)\right]^nh(s)= a^{-n}, \] so
\[\gamma(s)=\int_B^{\alpha s+B} a^{-1}\alpha^{-1}c^{-\frac{1}{n}} h^{-\frac{1}{n}}(t)dt.\]
Note that when $w$ satisfies \hyperlink{conditionK}{condition (K)}, i.e., $w(t)\sim_{+\infty} th(\log t)$ and $\int_1^\infty h^{-\frac{1}{n}}(t)dt<+\infty$, then $\gamma(s)<+\infty$. Thus, by definition, $|\psi-V_\theta|$ is bounded.

We can choose $B>0$ so that $\gamma'(0)\leq 1$, i.e., $h^{\frac{1}{n}}(B)\geq a^{-1}c^{-\frac{1}{n}}$ since $h(t)\to\infty$ as $t\to+\infty$.
We have on $\{\log f>-\alpha v+B\}$, 
{\[\MA_\theta(\psi)\geq a^{-n}fdV=a^{-n}\MA_\theta(\varphi).\]} 
Therefore, we have
\[\MA_\theta(\varphi)\leq a^{n}\MA_\theta(\psi)+ e^Be^{-\alpha v}dV. \]  
\end{proof}

Finally, we can prove Theorem~\ref{thm: main}.
\begin{proof}[Proof of Theorem~\ref{thm: main}]



As follows from Lemma~\ref{lem: reduceLp}, since $w$ satisfies  \hyperlink{conditionK}{condition (K)}, there exists $\psi\in\PSH(X,\theta)$ with minimal singularities and
$g\in L^2(dV)$ such that
\[ \MA_\theta(\varphi)\leq 2^{-n}\MA_\theta(\psi)+g dV.\]
As in Lemma~\ref{lem: reduceLp}, $\|g\|_2\leq C(a,\|f\|_w, C_\alpha)$.
Invoking Proposition~\ref{prop: moreestimate}, there exists a constant $C>0$ depending on $A$, $\omega_X$, $dV$, $n$, $a$, $\|e^{-\alpha v}\|_2$ and $\|f\|_w$ such that
\[ \varphi\geq\frac{\psi+V_\theta}{2}-C\geq V_\theta-C'.\]
\end{proof}

We end this section by stating a slight generalization of Theorem \ref{thm: main} and ~Proposition \ref{prop: moreestimate} as follows.
\begin{proposition} Let $\psi\in\PSH(X,\theta)$ such that $\sup_X\psi=0$, $a\in [0,1)$. Let $0\leq  f\in L^w(dV)$, where $w$ satisfies \hyperlink{conditionK}{condition (K)}.  
    Let $\varphi\in\mathcal{E}(X,\theta)$ be such that $\sup_X\varphi=0$ and 
    \[\MA_\theta(\varphi)\leq a^n\MA_\theta(\psi)^n +fdV.\]
    Then there exists a constant $C>0$ depending on $dV$, $A$, $\omega_X$, $n$, $p$, $a$ and $\|f\|_w$ such that 
    \begin{equation*}
        \varphi\geq a\psi+(1-a)V_\theta-C.
    \end{equation*}
\end{proposition}
\begin{proof}
    The proof immediately follows from Proposition~\ref{prop: moreestimate} and Theorem~\ref{thm: main}.
\end{proof}

\subsection{$L^\infty$-estimates in families}
The same argument as previous section therefore produces a slight refinement of the main theorem; cf.~\cite[Theorem 1.9]{GGZ23-families}.

\begin{theorem}\label{thm: family-C0-1}
    Let $(X,\omega_X)$ be a compact K\"ahler manifold. Let $\theta$ be a smooth closed (1,1) form representing a big cohomology class. Assume that $\theta\leq A\omega_X$ for some $A>0$.
    Let $\nu$ and $\mu= fd\nu$ be a probability measures on $X$.
    Assume the following assumptions are satisfied:
\begin{itemize}
    \item[(A1)] there exist $\alpha>0$ and $C_\alpha$ such that for all $u\in\PSH(X,\theta)$,
    \[\int_Xe^{-\alpha(u-\sup_X u)}d\nu\leq C_\alpha; \]
    \item[(A2)]  $0\leq f\in L^w(d\nu)$ with $w$ satisfying \hyperlink{conditionK}{condition~(K)}, $f>0$ $\nu$-almost everywhere, and there exists $C_w>0$ such that $\int_Xw\circ fd\nu\leq C_w$.
\end{itemize}
    Then there exists a unique solution $\varphi\in\mathcal{E}(X,\theta)$ with minimal singularities to the following equation \begin{equation}\label{cmae-big2}
    \MA_\theta(\varphi)=\mu,\quad \sup_X\varphi=-1.
\end{equation}
    More precisely, there is a constant $C$ depending on $A$, $dV$, $\omega_X$, $n$, $\alpha$, $C_\alpha$, $w$ and $C_w$ such that 
    \[ \varphi\geq V_\theta-C.\]
\end{theorem}
\begin{proof}
    We note that Skoda's integrability (or Tian's) theorem (Theorem~\ref{thm: Skoda/Tian}) is replaced by Assumption (A1). We proceed the same as in the proof of Theorem~\ref{thm: main}.
\end{proof}

One can obtain similarly the family version of uniform estimates for solutions to complex Monge-Amp\`ere equations.
\begin{theorem} Let $\mathfrak{X}$ be an irreducible and reduced complex K\"ahler space. 
    Let $\pi:\mathfrak{X}\to\mathbb D$ be a proper holomorphic surjective map such that 
    each fiber $X_t:=\pi^{-1}(t)$ is an $n$-dimensional compact K\"ahler manifold for $t\in\mathbb{D}^*$, and $X_0$ is a
reduced irreducible compact K\"ahler space. Fix $\omega_\mathfrak{X}$ a K\"ahler form on  $\mathfrak{X}$ and set $\omega_t=\omega_\mathfrak{X}|_{X_t}$.  
Let $\{\theta_t\}_{t\in \mathbb{D}_{1/2}}$ be a family of smooth closed big (1,1) forms such that $ \theta_t\leq A\omega_{\mathfrak{X}}$ for $A>0$. Let $(\mu_t)$ be a family of probability measures on $\mathfrak{X}$ such that $\mu_t=f_t\omega_t^n$, where $f_t\in L^{w}(X_t,dV_t)$ with $w$ satisfying \hyperlink{conditionK}{condition (K)}. 
Let $\varphi_t\in\mathcal{E}(X,\theta_t)$ be a solution to
\[ \MA_{\theta_t}(\varphi_t)=\mu_t,\qquad\sup_{X_t}\varphi_t=-1.\]
There exists a constant $C_\mathcal{K}$ such that the following
inequalities 
\[\varphi_t- V_{\theta_t}\geq -C \] where $C$ depends on $\omega_{\mathfrak{X}}$, $n$, $A$, $w$, and an upper bound for $\|f_t\|_w$, is independent of $t\in \mathbb{D}_{1/2}$. 
\end{theorem}
\begin{proof}
    From~\cite[Conjecture 3.1]{GGZ23-families} and \cite[Corollary 4.8]{Ou22-admissible}, we can show that Assumption (A1) in Theorem~\ref{thm: family-C0-1} is satisfied; cf.~\cite[Theorem 1.8]{guedj2025-green}. That is, there exists a constant $C_0>0$ such that
    \[ \sup_X u_t-C_0\leq \frac{1}{V_t}\int_{X_t}u_t\omega_t^n\leq \sup_X u_t\] for every $u_t\in\PSH(X,\theta_t)$. Applying~\cite[Theorem 2.9]{GGZ23-families}, we can find  constants $\alpha>0$ and $C_\alpha>0$ such that
    \[\int_Xe^{-\alpha(u_t-\sup_{X_t} u_t)}\omega_t^n\leq C_\alpha. \]The proof follows from Theorem~\ref{thm: family-C0-1}.
\end{proof}
  
  \subsection{Monge-Amp\`ere equations with prescribed singularities}  
   \begin{theorem}\label{thm: prescribed}
       Let $(X,\omega_X)$ be a compact K\"ahler manifold. Let $\theta$ be a smooth closed (1,1) form representing a big cohomology class. Assume that $\theta\leq A\omega_X$ for some $A>0$. Let $\phi\in\PSH(X,\theta)$ be a model potential.
    Let $\mu= fdV$ be a probability measure on $X$, where $dV$ is the Lebesgue measure, and $f\in L^w(dV)$ with density $w$ satisfying \hyperlink{conditionK}{condition~(K)}. 
    Then there exists a unique solution $\varphi\in\mathcal{E}(X,\theta,\phi)$  to the following equation \begin{equation}\label{cmae-big-pres}
    \MA_\theta(\varphi)=\mu,\quad \sup_X\varphi=-1,\quad \varphi\simeq \phi.
\end{equation}
    More precisely, there is a constant $C$ depending on $\omega_X$, $A$, $dV$, $n$, $w$, and $\|f\|_w$ such that 
    \[ \varphi\geq \phi-C.\]
   \end{theorem}
   \begin{proof} It follows from~\cite{darvas2021log} that
       there exists a solution $\varphi\in\mathcal{E}(X,\theta,\phi)$ such that $\MA_\theta(\varphi)=\mu$ and $\sup_X\varphi=-1$. We proceed the same way as in Theorem~\ref{thm: main} (replacing $V_\theta$ with $\phi$) to show that $\varphi\geq \phi-C$ with $C>0$ under control.
   \end{proof}
   \section{Moser-Trudinger and Brezis-Merle type inequalities}\label{sect: inequality}
We always assume $(X,\omega_X)$ is a compact K\"ahler manifold of dimension $n$. Assume $\theta$ is a closed smooth (1,1) form representing a big cohomology class and $\theta\leq A\omega_X$ for a fixed constant $A>0$. 
  
  We establish the Moser-Trudinger type inequality in terms of energy.
\begin{theorem}\label{thm: Moser-Trudinger}
Let $\varphi\in\mathcal{E}^p(X,\theta)$ be such that $\sup_X\varphi=-1$. Then there exists $c,C>0$ depending on $\omega_X$, $A$, $n$, $p$, and $\alpha$-invariant such that
\[\int_X\exp[c|E_p(\varphi)|^{-\frac{1}{n}}(V_\theta-\varphi)^{\frac{n+p}{n}} ]\omega_X^n\leq C. \]
\end{theorem}
This theorem is proved in \cite[Theorem 7]{GuoPhongTong23-estimates} and~\cite[Theorem 2.1]{di2021finite} for the K\"ahler case (which extends the result of Berman and Berndtsson~\cite{BermanBerndtsson22-Moser} for $p=1$), while the case of big cohomology class is treated in~\cite[Theorem 2.11]{DiNezza22-geodesic} and \cite[Theorem 3.5]{DiNezza24-entropy}. We also refer to~\cite{WangWangZhou20-Trudinger} for the local version.

   \begin{proof} We set $\MA_\theta(\varphi)=f\omega_X^n$.
   It follows from~\cite{boucksom2010monge,berman2013variational} that there exists a function $v\in\mathcal{E}(X,\theta)$ solving
   \begin{equation*}
       \MA_\theta(v)=\frac{(V_\theta-\varphi)^pf}{E_p(\varphi)}\omega_X^n,\quad\sup_X v=-1.
   \end{equation*}
Then by the same arguments as in Lemma~\ref{lem: key} (with $r=p$, $s=0$) we have
\[c(n,p)\left(\frac{V_\theta-\varphi}{(2E_p(\varphi))^{\frac{1}{n+p}}} \right)^{\frac{n+p}{n}}\leq V_\theta- v+C(n,p)E_p(\varphi)^{\frac{1}{p}}. \]
The rest of the proof follows the argument in~\cite[Theorem 7]{GuoPhongTong23-estimates}. We set $$U_{\kappa}:=\{V_\theta-\varphi\leq \kappa(2E_p(\varphi))^{\frac{1}{p}}\},$$ where $\kappa:=(2C)^{\frac{n}{n+p}} c^{-\frac{n}{n+p}}$. Then on $X\setminus U_\kappa$, we have
$$ \frac{1}{2}c(n,p)\left(\frac{V_\theta-\varphi}{(2E_p(\varphi))^{\frac{1}{n+p}}} \right)^{\frac{n+p}{n}}\leq V_\theta-v,$$  
and on $U_\kappa$, we have
\[ c(n,p)\left(\frac{V_\theta-\varphi}{(2E_p(\varphi))^{\frac{1}{n+p}}} \right)^{\frac{n+p}{n}}\leq {c(n,p)\kappa^{\frac{p}{n}}(V_\theta-\varphi)}.\]
Multiplying by $\min(c^{-1}\kappa^{-\frac{p}{n}},1/2)\alpha$ and integrating both sides, we obtain
\begin{align*}
    \int_X\exp &\left\{c'(n,p)\alpha  E_p(\varphi)^{-\frac{1}{n}} (V_\theta-\varphi)^{\frac{n+p}{n}}\right\}\omega_X^n\\
  &\qquad   \leq\int_{U_\kappa}{e^{\alpha (V_\theta-\varphi)}}\omega_X^n+ \int_{X\setminus U_\kappa}e^{\alpha(V_\theta- v)}\omega_X^n\leq C_\alpha.
\end{align*}
   \end{proof}

We establish the following Brezis-Merle-type inequalities for Monge-Amp\`ere potentials in big cohomology classes. 
    \begin{theorem}\label{thm: Brezis-Merle}  Let $\mu= f\omega_X^n$ be a probability measure on $X$ and $f\in L^w(dV)$, where $w(t)=t(\log(1+t))^p$ for $p\geq 0$. Assume that the p-Nash entropy $\textrm{Ent}_p(f)\leq B$, defined in~\eqref{eq: entropy}.
       Let $\varphi\in\mathcal{E}(X,\theta)$ be a solution to \eqref{eq: cmae-big}. Then there exist constants $c>0$ depending on  $\omega_X$, $n$, $p$, $A$,
       and $C>0$ depending on  $\omega_X$, $n$, $p$, $A$, $B$, $c$, and $\|f\|_w$ such that
       \begin{itemize}
           \item For $p\in [0,n)$, we have
       \begin{equation}\label{eq: Moser-Trudinger}
           \int_X e^{c(V_\theta-\varphi)^{\frac{n}{n-p}}}\omega_X^n\leq C.
       \end{equation} 
       \item for $p\geq  n$, we have for any $N>0 $ 
       \[ \int_Xe^{c(V_\theta-\varphi)^N}\omega_X^n\leq C.\]    
       \end{itemize}    
       Moreover, we also have the energy-like inequalities 
       \[\int_X(V_\theta-\varphi)^{r} f\omega_X^n\leq C, \]
       where $r=\frac{np}{n-p}$ if $p\in(0,n)$, and $r=N$ any positive constant if $p\geq n$.
   \end{theorem}
   The local version of this theorem can be found in~\cite{WangWangZhou20-Trudinger}, using a PDE approach, while the proof of Berman and Berndtsson~\cite{BermanBerndtsson22-Moser} uses induction on dimension and the ``thermodynamical formalism”~\cite{Berman13-thermodynamical}. The case $p=1$ is treated in \cite{di2021finite,DiNezza24-entropy} as a corollary of the Moser-Trudinger inequality. Actually, their proof also holds for any $p>0$; cf. Remark~\ref{rem: Brezis-Merle} below. The proof we provide here follows from generalized Ko\l odziej's $L^\infty$-estimate (cf. 
   Proposition~\ref{prop: moreestimate} and Lemma~\ref{lem: reduceLp}). The idea goes back to ~\cite{Qiao2025-sharp,liu2024relative}.

\begin{proof}
It is shown in~\cite[Theorem A]{boucksom2010monge} that there exists $v\in\mathcal{E}(X,\theta)$ that solves
\begin{equation}
    \MA_\theta(v)=cw\circ f dV,\quad\sup_Xv=0,\quad\text{where}\; c^{-1}=\int_Xw\circ f dV.
\end{equation}
From Lemma~\ref{lem: reduceLp}, we obtain
\[\MA_\theta(\varphi)\leq 2^{-n}\MA_\theta(\psi)+e^{-\alpha v}dV \]
where $\alpha=\alpha(n,A\omega_X)/2$, $\psi=V_\theta-\gamma(V_\theta-v)$, defined in~\eqref{eq: psi}, and 
\[\gamma(s)=2\alpha^{-1}\|f\|_w^{\frac{1}{n}}\int_B^{\alpha s+B}h^{-\frac{1}{n}}(s)dt,\quad\text{with}\; w(t)=th(\log t). \]
It follows from Proposition~\ref{prop: moreestimate}, that  \[\varphi\geq \frac{\psi+V_\theta}{2}-C\geq V_\theta-\frac{1}{2}\gamma(V_\theta-v)-C, \]
for $C>0$ depending on $dV$, $A$, $\omega_X$, $n$, $\alpha$, $\|e^{-\alpha v}\|_{L^2}$.

We consider three cases:
\begin{enumerate}
    \item If $w(t)=t(\log t)^p$ for $p\in(0,n)$, then $$\gamma(s)=2\alpha^{-1}\|f\|_w^{\frac{1}{n}}\frac{n}{n-p}\left[(\alpha s+B)^{1-\frac{p}{n}}-B^{1-\frac{p}{n}} \right].$$ Therefore, there exist $c>0$ and $C>0$ depending on $c$, $\alpha$, $B$, $\|f\|_w$, such that
    \[\int_X \exp[c(V_\theta-\varphi)^{\frac{n}{n-p}}]dV\leq C_1\int_X e^{\alpha(V_\theta-v)}dV\leq C_1\int_Xe^{-\alpha v}dV\leq C_2,\]  
where the last inequality follows from Skoda's integrability theorem.

    \item If $w(t)=t(\log t)^p$ for $p=n$, then $\gamma(s)=2\alpha^{-1}\|f\|_w^{\frac{1}{n}}[\log (\alpha s+B)-\log B]$. The desired estimate follows.
    \item If $p=0$, then as follows from Lemma~\ref{lem: Young-fuction}, we can find a Young function $w(t)=tg(t)$ such that $\int_Xw(f)dV<+\infty$ where $g(t)\nearrow+\infty$ as $t\to+\infty$. In this case $\gamma(s)\leq \alpha s+C_1$ for some $C_1>0$. We also get the desired estimate.
\end{enumerate}
For the last statement, we apply the H\"older-Young inequality to obtain  
\begin{align*}
    \int_X c^p(V_\theta-\varphi)^{pq}f\omega^n_X\leq \int_Xw\circ f\omega_X^n+C_p\int_X e^{c(V_\theta-\varphi)^q}\omega_X^n,
\end{align*} where $q=\frac{n}{n-p}$ if $p\in(0,n)$ and $q=N$ any positive exponent if $p=n$. This completes the proof.
\end{proof}

\begin{remark}
     For $p>0$, we can apply the Lemma~\ref{lem: key} with $r=\frac{np}{n-p}>0$ ($r$ is any positive
number if $p=n$) to infer that on $\Omega_s=\{\varphi<V_\theta-s\}$,
\[\frac{(V_\theta-\varphi-s)}{A_s^{\frac{1}{n+r}}}\leq C(-v_s+CA_s^{\frac{1}{r}})^{\frac{n}{n+r}}, \] for $C>0$ depending on $\omega_X$, $A$, $n$, $p$, where $A_s=\int_{\Omega_s}(V_\theta-\varphi-s)^rf\omega_X^n$ and $v_s$ solves
\[\MA_\theta(v_s)=\frac{\mathbf{1}_{\Omega_s}(V_\theta-\varphi-s)^r}{A_s}f\omega_X^n,\quad\sup_X v_s=0. \]
We can proceed with the arguments as in Steps 3, 4 of~\cite[Theorem 1]{GuoPhong24-entropy} (see also~\cite[Proposition 5.1]{Qiao2025-sharp}) to finish the proof.
\end{remark}

\begin{remark}\label{rem: Brezis-Merle}
  When $p>0$, the Brezis-Merle inequalities (Theorem~\ref{thm: Brezis-Merle}) are consequences of the Moser-Trudinger inequality (Theorem~\ref{thm: Moser-Trudinger}). The proof follows from~\cite[Theorem 3.4]{di2021finite}.

 Indeed, if $p\in(0,n)$, we claim that $\mathcal{E}^r(X,\theta)\subset L^r(X,\mu)$ with $r=\frac{np}{n-p}$. Fix $v\in\mathcal{E}^r(X,\theta)$ with $\sup_X v=-1$. By the Moser-Trudinger inequality (Theorem~\ref{thm: Moser-Trudinger}), since $1+\frac{r}{n}=\frac{n}{n-p}$ we have 
 \[\int_Xe^{c(V_\theta-v)^{\frac{n}{n-p}}}\omega_X^n<\infty. \]
 We next apply the H\"older-Young inequality to obtain  
\begin{align*}
    \int_X c^p(V_\theta-v)^{r}\MA_\theta(\varphi)\leq \int_Xw\circ f\omega_X^n+C_p\int_X e^{c(V_\theta-v)^{\frac{n}{n-p}}}\omega_X^n<\infty,
\end{align*} which proves the claim by~\cite[Theorem 6.2]{darvas2025relative}. Hence, $\varphi\in\mathcal{E}^r(X,\theta)$.

Using the H\"older–Young inequality again, we obtain
\begin{align*}
    \int_X c^p(V_\theta-\varphi)^{r}|E_r(\varphi)|^{-\frac{p}{n}}f\omega_X^n\leq \int_Xw\circ f\omega_X^n+C_p\int_X e^{c(V_\theta-v)^{\frac{n}{n-p}}|E_r(\varphi)|^{-\frac{1}{n}}}\omega_X^n\leq C_1,
\end{align*} thanks to the Moser-Trudinger inequality, where $C_1>0$ depends on $\omega_X$, $n$, $p$ and $\|f\|_w$. This implies that $c^pE_r(\varphi)^{1-\frac{p}{n}}\leq C_1$, so $E_r(\varphi)\leq C_2$.     By the Moser-Trudinger inequality again, we obtain
\begin{align*}
    \int_Xe^{c'(V_\theta-\varphi)^\frac{n}{n-p}}\omega_X^n\leq C_3,
\end{align*} where $c'=cC_2^{-\frac{1}{n}}$. 

When $p=n$, we choose $r=nN$ for any positive constant $N$. The proof follows the same as above.

\end{remark}

	\bibliographystyle{alpha}
	\bibliography{bibfile}	
	
\end{document}